%% file: main.tex
\documentclass{amsart}
\usepackage{amssymb}
\usepackage{graphicx}
\usepackage{hyperref}
\usepackage{textcase}
\usepackage[foot]{amsaddr}

\numberwithin{equation}{section}
\numberwithin{figure}{section}
\theoremstyle{plain}
\newtheorem{thm}{\protect\theoremname}[section]
\theoremstyle{plain}
\newtheorem{lem}[thm]{\protect\lemmaname}
\theoremstyle{definition}
\newtheorem{defn}[thm]{\protect\definitionname}
\theoremstyle{plain}
\newtheorem{cor}[thm]{\protect\corollaryname}
\theoremstyle{remark}
\newtheorem{rem}[thm]{\protect\remarkname}

\providecommand{\corollaryname}{Corollary}
\providecommand{\definitionname}{Definition}
\providecommand{\lemmaname}{Lemma}
\providecommand{\remarkname}{Remark}
\providecommand{\theoremname}{Theorem}

\allowdisplaybreaks
\include{base_definitions}

\include{definitions}

\include{biographical}
\begin{document}

\title[\shortpapertitle]{\papertitle}
\author{\miekname}
\address{\miekaffil}
\email{\miekemail}

\subjclass[2010]{Primary: \primarymsccodeA, Secondary: \secondarymsccodeA, \secondarymsccodeB}
\begin{abstract}\paperabstract\end{abstract}
\maketitle

\input{CONTENT}

\bibliographystyle{amsalpha}
\bibliography{bibliography}

\end{document}

%% file: base_definitions.tex
\global\long\def\abs#1{\left|#1\right|}%

\global\long\def\set#1#2{\left\{  \vphantom{#1}\vphantom{#2}#1\right.\left|\ #2\vphantom{#1}\vphantom{#2}\right\}  }%

\global\long\def\parenth#1{\left(#1\right)}%

\global\long\def\N{\mathbb{N}}%

\global\long\def\R{\mathbb{R}}%

\global\long\def\cal#1{\mathcal{#1}}%

%% file: definitions.tex
\global\long\def\codecenter{\textup{center}}%

\global\long\def\petals{\textup{petals}}%

\global\long\def\radii{\textup{radii}}%

\global\long\def\coronas{\textup{coronas}}%

\global\long\def\code{\textup{code}}%

\global\long\def\ang{\textup{angles}}%

\global\long\def\len{\textup{length}}%

\global\long\def\codes{\mathfrak{C}}%

\global\long\def\angles{\mathfrak{A}}%

%% file: biographical.tex
\newcommand{\papertitle}{The number of configurations of radii that can occur in compact packings of the plane with discs of n sizes is finite}
\newcommand{\shortpapertitle}{Configurations of radii that can occur in compact disc packings}

\newcommand{\paperabstract}{%
By a compact packing of the plane by discs, $P$, we mean a collection
of closed discs in the plane with pairwise disjoint interior so that,
for every disc $C\in P$, there exists a sequence of discs $D_{0},\ldots,D_{m-1}\in P$
so that each $D_{i}$ is tangent to both $C$ and $D_{i+1\mod m}.$ 

We prove, for every $n\in\N$, that there exist only finitely many
tuples $(r_{0},r_{1},\ldots,r_{n-1})\in\R^{n}$ with $0<r_{0}<r_{1}\ldots<r_{n-1}=1$
that can occur as the radii of the discs in any compact packing of
the plane with $n$ distinct sizes of disc.
}

\newcommand{\miekname}{Miek Messerschmidt}
\newcommand{\miekemail}{miek.messerschmidt@up.ac.za}
\newcommand{\miekaffil}{%
Department of Mathematics and Applied Mathematics, University of Pretoria,
Private bag X20,
Hatfield 0028,
South Africa}

\newcommand{\primarymsccodeA}{05B40} %Combinatorial aspects of packing and covering
\newcommand{\secondarymsccodeA}{52C26} %Circle packings and discrete conformal geometry
\newcommand{\secondarymsccodeB}{52C15} %Packing and covering in $2$ dimensions

\newcommand{\paperkeywords}{Circle packing, Compact packing}

%% file: CONTENT.tex
\section{Introduction}

For any set of discs $Q$ in the plane, with $D\in Q$ and $r_{D}\in\R$
denoting the radius of a disc $D\in Q$, we define $\radii(Q):=\{r_{D}:D\in Q\}$.
With $m\in\N$ by a \emph{corona} we mean a collection of closed discs
in the plane $C,D_{0},D_{1},\ldots,D_{m-1}$ that have pairwise disjoint
interiors and are labeled in such a way so that for every $i\in\{0,\ldots,m-1\}$
the disc $D_{i}$ is tangent to both the discs $C$ and $D_{i+1\mod m}.$
We call $C$ \emph{the center }of the corona and $D_{0},D_{1},\ldots,D_{m-1}$
\emph{the petals }of the corona. 

By a \emph{compact packing (of the plane by discs)} \texttt{$P$}
we mean a collection of closed discs in the plane with pairwise disjoint
interiors, so that every disc $C\in P$ is the center of a corona
with petals from $P$ (e.g., Figure~\ref{fig:example0}). The set
of all maximal (with respect to inclusion) coronas that occur in a
compact packing $P$ is denoted by $\coronas(P).$ 

\begin{figure}
\includegraphics[width=1\textwidth]{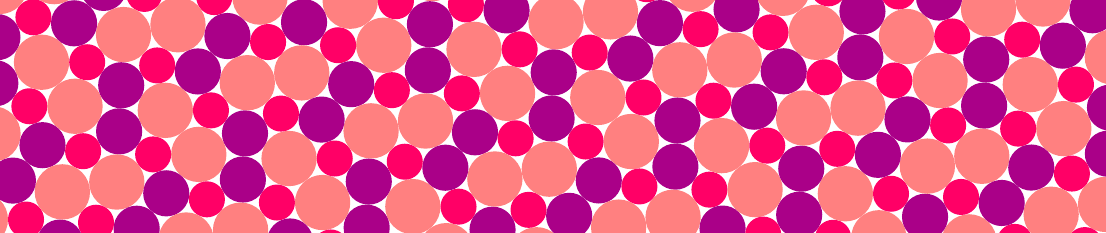}\caption{\label{fig:example0}An example of a compact $3$-packing \cite[Appendix~A. Example (53)]{Fernique2021}.}
\end{figure}

With $P$ a compact packing, in this paper, we restrict ourselves
always to the case $\abs{\radii(P)}<\infty$. For $n\in\N$, we will
say $P$ is a \emph{compact $n$-packing} if $\abs{\radii(P)}=n$.
We define the set
\[
\Pi_{n}:=\set{\radii(P)}{\begin{array}{c}
P\text{ a compact \ensuremath{n}-packing with }\\
\max\,\radii(P)=1
\end{array}}.
\]
In this paper we prove the following result:
\begin{thm}
\label{thm:main-thm}For all $n\in\N$, the set $\Pi_{n}$ is finite.
\end{thm}

Before discussing our approach in this paper toward proving Theorem~\ref{thm:main-thm},
we first make some brief historical remarks. 

Since the hexagonal packing is the only compact packing with one size
of disc we have $\abs{\Pi_{1}}=1$. This was likely known to the ancients.
The equality $\abs{\Pi_{2}}=9$ was only established in 2006 by Kennedy
in \cite{Kennedy2006}. Of the nine elements in $\Pi_{2}$, some were
known earlier, with seven appearing in Fejes Toth's 1964 book \cite[p.185--187]{Toth1964}
and a further one appearing in 1993 in \cite{Likos1993}. 

Determining whether or not $\abs{\Pi_{3}}$ is finite and computing
its value took a number of years after publication of \cite{Kennedy2006}
and proved to be a tantalizing problem in its own right. The bound
$\abs{\Pi_{3}}\leq13617$ was established in \cite{Messerschmidt2020}.
Working roughly at the same time, but independently from the author,
a better approach was developed by Fernique, Hashemi and Sizova who
in \cite{Fernique2021} showed that $\abs{\Pi_{3}}=164$. Based on
the values of $\abs{\Pi_{1}},\abs{\Pi_{2}}$ and $\abs{\Pi_{3}}$,
Connelly and Pierre made the conjecture in a previous version of the
preprint \cite{Connelly2021} that the sequence $(\abs{\Pi_{n}})$
is the OEIS sequence A086759 \cite{oeis}.

Analogous work, studying compact packings of three dimensional space
with spheres of two and three sizes, was performed by Fernique respectively
in \cite{FerniqueTwoSpheres2021} and \cite{FerniqueThreeSizes}.
Surprisingly, the situation in three dimensions is not quite as rich
as the two dimensional case (at least for two and three sizes of spheres),
in that all the compact packings of three dimensional space with two
and three sizes of spheres arise through merely filling the interstitial
holes in a close-packing of spheres with unit radius. This is not
the case for the compact packings of discs of the plane with two and
three sizes of discs as is evident in \cite{Kennedy2006} and \cite{Fernique2021}. 

\medskip

As the ideas in the current paper are partially informed by those
employed in analyzing $\Pi_{3}$, we briefly describe the process
involved in establishing $\abs{\Pi_{3}}<\infty$ along with some of
the difficulties that arise when analyzing $\Pi_{3}$ that do not
exist in determining $\Pi_{2}$. 

Define 
\[
\cal P_{3}:=\set P{P\text{ a compact \ensuremath{3}-packing with }\max\:\radii(P)=1}.
\]
 As there occur exactly three sizes of discs in any packing $P\in\cal P_{3}$,
for convenience we refer to the disc sizes as \emph{small}, \emph{medium
}and \emph{large}. To any corona one may associate an abstract \emph{coronal
code} (see Sections~\ref{subsec:Coronal-codes} and~\ref{subsec:Canonical-labeling}
for exact definitions). The length of the code determined by a corona
is exactly equal to the length of the sequence of petals in the corona.
Mimicking Kennedy's argument from \cite{Kennedy2006}, it is not difficult
to see that for any compact $3$-packing $P\in\cal P_{3}$, all codes
determined by coronas in $P$ with small centers have a universal
bound of six on their length. 

A first difficulty arises: That of determining a universal bound on
the length of codes associated to any coronas with medium centers
in any compact $3$-packing $P\in\cal P_{3}$. Failing this, there
\emph{could }exist infinitely many coronas with medium centers that
occur across compact $3$-packings in $\cal P_{3}$, all with distinct
associated coronal codes. This would imply $\abs{\Pi_{3}}=\infty$.
A crucial observation in showing that this is not the case is the
result \cite[Lemma~6.1]{Fernique2021} paraphrased here: 
\begin{lem}
\label{lem:Fernique-lemma}There exists a universal constant $K>0$
so that, for any compact $3$-packing $P$ with $\radii(P)=\{r_{0},r_{1},r_{2}\}$
satisfying $0<r_{0}<r_{1}<r_{2}=1$, we have 
\[
0<K\leq\frac{r_{0}}{r_{1}}.
\]
The largest value of $K$ can take on is $\min\{r:\{r,1\}\in\Pi_{2}\}.$
\end{lem}

The proof of Lemma~\ref{lem:Fernique-lemma} proceeds through leveraging
knowledge of $\Pi_{2}$ to determine a universal lower bound $K$
for all possible ratios of small-to-medium radii that can occur in
any compact $3$-packing from $\cal P_{3}$. This is crucial in placing
a universal bound on the length of the coronal codes associated to
coronas with medium center that can occur across all compact $3$-packings.
This universal bound allows one to enumerate all of the finitely many
pairs of codes that can possibly be associated to coronas with small
and medium centers occurring in any compact $3$-packing from $\cal P_{3}$.
This idea returns in Section~\ref{sec:The-bootstrapping-theorem}
of the current paper where we prove a more general version of Lemma~\ref{lem:Fernique-lemma}
in Lemma~\ref{lem:bootstrap}.

A second difficulty now arises: There exist infinitely many coronas
containing discs of three different radii that are not congruent (modulo
rigid motions and scaling), but who nevertheless determine the same
coronal code (see Figure~\ref{fig:two-coronas-same-code}). By inflating
and deflating discs in a compact packing, one could thus conceivably
obtain different compact packings yet with the same combinatorics
(see for example \cite{ConnellyGortler} where such techniques are
used to continuously deform finite compact packings whose contact
graphs are related by edge flips into oneanother). For this reason,
there could exist infinitely many more elements in $\Pi_{3}$ than
the finite set all possible pairs of coronal codes that can be associated
to coronas with small and medium centers in any compact $3$-packing
from $\cal P_{3}$. One then needs to show that this is not the case:
In any compact $3$-packing $P\in\cal P_{3}$ there exists a pair
of coronas with small and medium centers so that their associated
coronal codes necessarily satisfy conditions (which we do not state
here) that uniquely determine the values in $\radii(P)$. This shows
that $\Pi_{3}$ has at most as many elements as pairs of codes that
satisfy the mentioned necessary conditions, of which there are only
finitely many. This idea also returns in the current paper (cf. Sections~\ref{sec:Fundamental-sets}
and~\ref{sec:unique-tight-realizations}).
\begin{figure}
\includegraphics[width=0.85\textwidth]{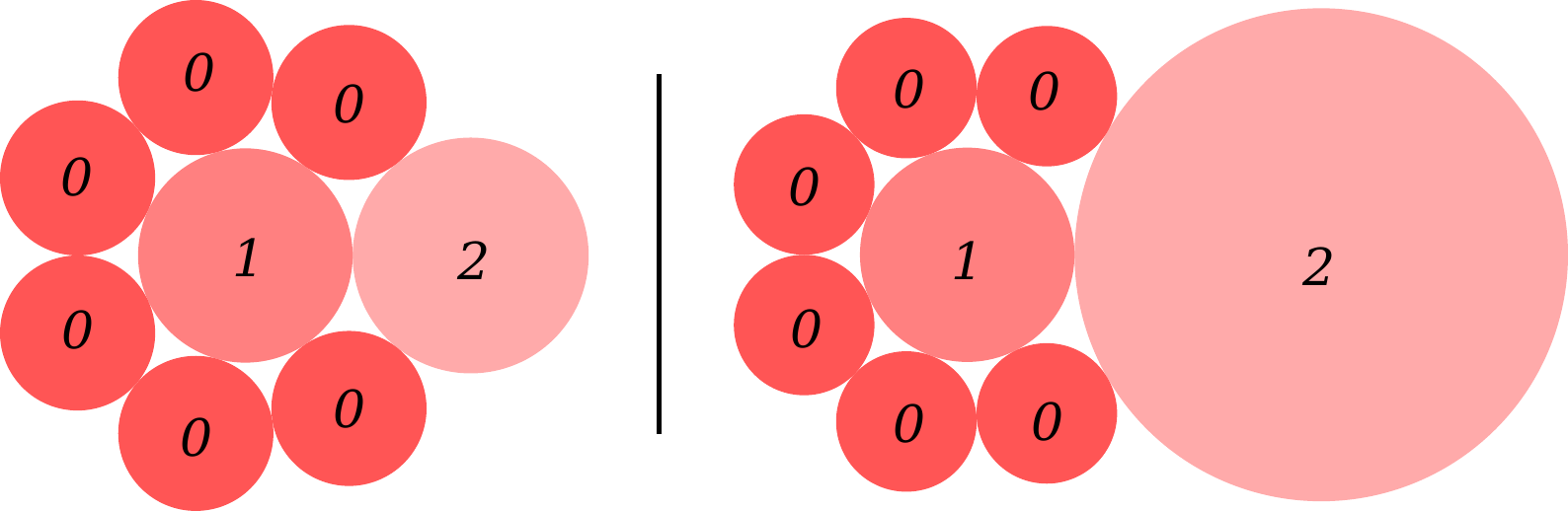}\caption{Two coronas with three sizes of discs that are not congruent (modulo
rigid motions and scaling) that determine the same coronal code 1:2000000.\label{fig:two-coronas-same-code}}
\end{figure}

The third and last difficulty arises in computing the exact values
of elements of $\Pi_{3}$. Elements of $\Pi_{3}$ are determined as
solutions to systems of trigonometric equations in two variables which
are obtained from the pairs of codes associated with coronas with
small and medium radii in some compact $3$-packing $P\in\cal P_{3}$.
Such systems of equations can be manipulated into systems of equations
containing radicals of rational expressions, and further into systems
of polynomial equations in two variables (cf. \cite[Section~8]{Messerschmidt2020}
and \cite[Sections~5 and~6]{Fernique2021}). The solutions of these
systems of polynomial equations can be determined as the roots of
single variable polynomials through computing appropriate Gr\"obner
bases of the occurring polynomials. Some of these computations can
become expensive (cf. \cite[Sections~3.5 and~3.6]{Fernique2021}).
As our methods in this paper do not require explicit computation of
elements in any of the sets $\Pi_{n}$ in order to prove Theorem~\ref{thm:main-thm},
the computational difficulties mentioned in this paragraph are not
encountered.

\medskip

We now turn to the approach taken in this paper to prove Theorem~\ref{thm:main-thm}.
The proof follows through strong induction and is somewhat informed
by the arguments and computations from \cite{Messerschmidt2020} and
\cite{Fernique2021}, but does not at all rely on any computer calculations. 

In Section~\ref{sec:Preliminaries} we define a simple symbolic language
so that we may analyze compact packings entirely abstractly. This
is done to facilitate the combinatorial arguments without the need
to refer to actual compact packings. In Sections~\ref{subsec:Coronal-codes}
and~\ref{subsec:Angle-symbols}, we define what we call \emph{coronal
codes }and \emph{angle symbols. }Coronal codes are abstract representations
of coronas, and angle symbols are abstract representations of the
angle formed at the center of a specific disc through connecting the
centers of three mutually tangent discs with disjoint interior (cf.
Figure~\ref{fig:angle-symbol}). For a coronal code we define what
we call its \emph{angle sum }which is the formal sum of the angle
symbols that represent the sum of the angles formed by connecting
the center of the central disc of a corona to the centers of its petals.

In Sections~\ref{subsec:Realizations} and~\ref{subsec:Canonical-labeling},
we describe what we call \emph{realizations }and what we call the
\emph{canonical labeling of a compact packing}. By a realization we
merely mean assigning concrete positive values to indeterminates in
a formal algebraic expression. For a compact packing, we describe
what we call the \emph{canonical labeling of the packing }(labeling
all discs with non-negative integers in increasing order of size)\emph{,
}along with its \emph{canonical realization} (mapping the label of
a disc to the radius of the disc). As such, all coronal codes from
coronas obtained from a canonically labeled compact packing are necessarily
such that their angle sums, when realized by the canonical realization,
evaluate to $2\pi$.

In Section~\ref{sec:Fundamental-sets} we define what we call a \emph{fundamental
set} of coronal codes. We continue to prove in Theorem~\ref{thm:packings-have-fundamental-sets}
that every canonically labeled compact packing is such that it necessarily
contains a set of coronas so that their associated coronal codes determines
a fundamental set.

In Section~\ref{sec:unique-tight-realizations} we prove Theorem~\ref{thm:fundamental-set-injective-G}.
This is an entirely abstract result which shows that any fundamental
set of coronal codes is such that a certain map from a set of certain
possible realizations to the tuple of realized angle sums is injective.
This allows us to prove Theorem~\ref{thm:packings-determine-exactly-one-realization},
which states that the coronal codes from a canonically labeled compact
packing determine a unique realization $\rho$ (i.e., the canonical
realization) for which these coronal codes' angle sums, when realized
by $\rho$, evaluate to $2\pi$.

Section~\ref{sec:The-bootstrapping-theorem} sees proof of what we
will call The Bootstrapping Lemma (Lemma~\ref{lem:bootstrap}). This
lemma is a generalization of Lemma~\ref{lem:Fernique-lemma}. It
allows us to relate ratios of radii that occur in certain different
realizations for the same fundamental set of coronal codes. This lemma
will be a crucial ingredient in a strong induction argument that forms
part of the main result.

In Section~\ref{sec:Essential-sets-of}, for $n\in\N$ with $n\geq2$,
we define what we call \emph{$n$-essential sets.} An $n$-essential
set is a fundamental set of coronal codes for which there exist two
monotone realizations so that the angle sum of all coronal codes in
the set, when realized by these two maps, the results respectively
lie in the intervals $(0,2\pi]$ and $[2\pi,\infty)$ (cf. Definition~\ref{def:Essential-sets}).
This definition allows for leveraging The Bootstrapping Lemma in a
strong induction argument to prove that, for all $n\in\N$ with $n\geq2$,
there exist at most finitely many $n$-essential sets (Corollary~\ref{cor:set-of-essential-sets-are-finite}).
We finally argue that every compact $n$-packing determines at least
one of the finitely many $n$-essential sets and that every $n$-essential
set determines at most one (perhaps no) element of $\Pi_{n}$. Therefore,
for every $n\in\N$ with $n\geq2$, cardinality of $\Pi_{n}$ is bounded
by the cardinality of the set of all $n$-essential sets (cf. Corollary~\ref{cor:main}).

We do remark that our method for showing $\abs{\Pi_{n}}<\infty$ is
non-constructive and therefore does not determine quantitative bounds
for $\abs{\Pi_{n}}$. As it stands, determining quantitative bounds
would require some significant computational resources (cf. Remark~\ref{rem:why-no-quantitative-bounds}).
For this reason, we opt to use bounds for the size of certain combinatorial
objects that are obvious, but definitely not optimal, as using tighter
bounds or even the exact values will not provide any advantage.

\section{\label{sec:Preliminaries}Preliminaries }

\subsection{\label{subsec:Coronal-codes}Coronal codes}

Let $S$ be any set of symbols and let $m\in\N$. With symbols $c,p_{0},p_{1},\ldots,p_{m-1}\in S$,
by a\emph{ coronal code (over $S$ of length $m$)}, we will mean
a formal string of the form 
\[
c:p_{0}p_{1}\ldots p_{m-1}.
\]
 In the coronal code $x:=c:p_{0}p_{1}\ldots p_{m-1}$, we call the
symbol, $c$, \emph{the center }of the coronal code $x$, and the
symbols, $p_{0},p_{1},\ldots,p_{m-1}$, the \emph{petals }of the coronal
code $x$. We define $\codecenter(x):=c$, $\petals(x):=\{p_{0},p_{1},\ldots,p_{m-1}\}$
and $\len(x):=m$. We will say two coronal codes $c:p_{0}p_{1}\ldots p_{m-1}$
and $d:q_{0}q_{1}\ldots q_{k-1}$ are equivalent if they have the
same length, $m=k$, have the same centers, $c=d$, and the formal
string $p_{0}p_{1}\ldots p_{m-1}$ equals some rotation and/or reflection
of the formal string $q_{0}q_{1}\ldots q_{m-1}$. We denote the set
of all equivalence classes of coronal codes over $S$ by $\codes(S)$.
We will not make explicit distinction between an equivalence class
of coronal codes in $\codes(S)$ and the members of the equivalence
class. 

\subsection{\label{subsec:Angle-symbols}Angle symbols}

Let $S$ be any set of symbols. For symbols $a,b,c\in S$ we define
the formal symbol 
\[
c_{b}^{a}:=\arccos\parenth{\frac{(c+a)^{2}+(c+b)^{2}-(a+b)^{2}}{2(c+a)(c+b)}}
\]
and call the symbol $c_{b}^{a}$ an \emph{angle symbol (over $S$)}.
In an arrangement of three mutually tangent discs, labeled $a,b$
and $c$, an angle symbol $c_{b}^{a}$ is defined to abstractly represent
the angle formed at the center of the disc $c$ by the line segments
connecting the center of the disc labeled $c$ to the centers of the
discs labeled $a$ and $b$ (cf. Figure~\ref{fig:angle-symbol}).

\begin{figure}
\includegraphics[width=0.2\textwidth]{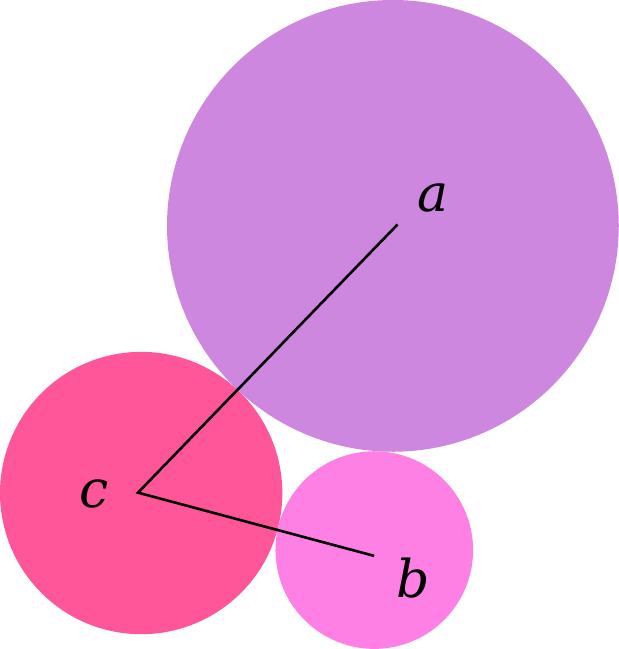}\caption{\label{fig:angle-symbol}The angle represented by the angle symbol
$c_{b}^{a}.$}
\end{figure}

We explicitly note that that elements in $S$ are \emph{always }to
be considered as indeterminate symbols in an angle symbol over $S$.
For an angle symbol $c_{b}^{a}$ we call $c$ the \emph{vertex }of
the angle symbol $c_{b}^{a}$ and we call $a,b$ the \emph{petals
}of the angle symbol $c_{b}^{a}$. We regard $c_{b}^{a}$ and $c_{a}^{b}$
as identical. We denote the set of all angle symbols over $S$ by
$\angles(S)$. For distinct symbols $a,b,c,d\in S$ we define the
following formal symbols:
\[
\begin{array}{ccc}
\begin{array}{rcl}
\partial_{a}c_{a}^{b}=\partial_{a}c_{b}^{a} & := & \frac{\sqrt{bc}}{(c+a)\sqrt{a}\sqrt{a+b+c}},\vspace{2mm}\\
\partial_{a}c_{a}^{c}=\partial_{a}c_{c}^{a} & := & \frac{c}{(c+a)\sqrt{a^{2}+2ac}},\vspace{2mm}\\
\partial_{a}c_{a}^{a} & := & \frac{2\sqrt{c}}{(c+a)\sqrt{2a+c}},\vspace{4mm}\\
\partial_{c}c_{a}^{b}=\partial_{c}c_{b}^{a} & := & -\frac{(a+b+2c)\sqrt{ab}}{(c^{2}+ab+ac+bc)\sqrt{c}\sqrt{a+b+c}},\vspace{2mm}\\
\partial_{c}c_{a}^{c}=\partial_{c}c_{c}^{a} & := & -\frac{a}{(c+a)\sqrt{a^{2}+2ac}},
\end{array} & \quad & \begin{array}{rcl}
\partial_{d}c_{b}^{a} & := & 0,\\
\partial_{d}c_{c}^{a} & := & 0,\\
\partial_{d}c_{a}^{a} & := & 0,\\
\partial_{d}c_{c}^{c} & := & 0,\\
\partial_{c}c_{c}^{c} & := & 0.
\end{array}\end{array}
\]
The symbols defined above are exactly the partial derivatives of the
involved angle symbols toward the stated indeterminate. For our purpose,
the exact values of the expressions above are not as important as
their signs, as will be seen in the proof of Lemma~\ref{lem:strict-inc-of-angle-sum-in-petal-directions}.
For $a\in S$, we will regard $\partial_{a}$ as an additive operator
on sums of angle symbols.

Let $x:=c:p_{0}p_{1}\ldots p_{m-1}\in\codes(S)$ be any coronal code.
We define the formal sum
\[
\alpha(x):=c_{p_{0}}^{p_{1}}+c_{p_{1}}^{p_{2}}+\ldots+c_{p_{m-2}}^{p_{m-1}}+c_{p_{m-1}}^{p_{0}}.
\]
The expression $\alpha(x)$ is called an \emph{angle sum }for the
coronal code $x$. We assume in this formal sum that addition commutes,
and therefore as a map, $\alpha$ is well defined on equivalence classes
in $\codes(S)$. The set of angle symbols that occur in coronal code
$x:=c:p_{0}p_{1}\ldots p_{m-1}\in\codes(S)$ is defined and denoted
as 
\[
\ang(x):=\{c_{p_{0}}^{p_{1}},c_{p_{1}}^{p_{2}},\ldots,c_{p_{m-2}}^{p_{m-1}},c_{p_{m-1}}^{p_{0}}\}\subseteq\angles(S).
\]

\subsection{\label{subsec:Realizations}Realizations}

Let $S$ be any set of symbols. By a \emph{realizer} of $S$ we mean
a map $\rho:S\to(0,\infty)$. With a realizer $\rho:S\to(0,\infty)$
and any (formal) arithmetic expression $E$ containing symbols from
$S$, by $E|_{\rho}$ we mean the expression $E$ with every symbol
$s\in S$ occurring in $E$ replaced by $\rho(s)$ and we will call
$E|_{\rho}$ the \emph{realization of $E$ by $\rho$}. E.g.,
\[
c_{b}^{a}|_{\rho}:=\arccos\parenth{\frac{(\rho(c)+\rho(a))^{2}+(\rho(c)+\rho(b))^{2}-(\rho(a)+\rho(b))^{2}}{2(\rho(c)+\rho(a))(\rho(c)+\rho(b))}}\in\R.
\]
We will say that $x\in\codes(S)$ is \emph{tightly realized by $\rho:S\to(0,\infty)$}
if $\alpha(x)|_{\rho}=2\pi$. With $C\subseteq\codes(S)$, we will
say that $C$ is \emph{tightly realized by $\rho$ }if all $t\in C$
are tightly realized by $\rho$.

\subsection{\label{subsec:Canonical-labeling}Canonical labeling and canonical
realizers}

Let $n\in\N$ and let $Q$ be any collection of discs in the plane
with $\abs{\radii(Q)}=n$. With $0<r_{0}<\ldots<r_{n-1}$ such that
$\radii(Q)=\{r_{0},\ldots,r_{n-1}\}$ we label every disc $D\in Q$
with the unique index $i\in\{0,\ldots,n-1\}$ attached to its radius
$r_{i}$ and we denote this label by $L_{D}$. We call this labeling
of discs \emph{the canonical labeling of $Q$}. The map $\rho:S\to(0,\infty)$
defined, for all $i\in S$, by $\rho(i):=r_{i}$ is called the \emph{canonical
realizer of $Q$.}

Let $P$ be a canonically labeled compact $n$-packing with $0<r_{0}<\ldots<r_{n-1}$
so that $\radii(P)=\{r_{0},\ldots,r_{n-1}\}$. With $S:=\{0,1,\ldots,n-1\}$,
we define the function $\code:\coronas(P)\to\codes(S)$ for any maximal
corona $\cal C\in\coronas(P)$ with center $C\in P$ and petals $D_{0},\ldots,D_{m-1}\in P$,
as 
\[
\code(\cal C):=L_{C}:L_{D_{0}}L_{D_{1}}\ldots L_{D_{m-1}}\in\codes(S).
\]
It is clear, by definition of a compact packing, that the set $(\code\circ\coronas)(P)\subseteq\codes(S)$
is necessarily tightly realized by the canonical realizer of \emph{$P$.}

\section{\label{sec:Fundamental-sets}Fundamental sets and fundamental sets
from compact packings}

In this brief section we introduce define the following structure
of coronal codes, and subsequently show that every compact $n$-packing
determines such a structure.
\begin{defn}
\label{def:fundamental}Let $n\in\N$ and $S:=\{0,1,\ldots,n-1\}$.
We say a non-empty set $C\subseteq\codes(S)$ is a \emph{fundamental
set (of coronal codes)} if $\{\codecenter(x):x\in C\}=\{0,\ldots,n-2\}$
and for every non-empty set $K\subseteq\{0,1,\ldots,n-2\}$ there
exists a subset $D\subseteq C$ satisfying $\{\codecenter(x):x\in D\}=K$
so that
\[
\parenth{\bigcup_{x\in D}\petals(x)}\backslash K\neq\emptyset.
\]
\end{defn}

\begin{thm}
\label{thm:packings-have-fundamental-sets}Let $n\in\N$ with $n>2$
and set $S:=\{0,\ldots,n-1\}$. Let $P$ be a canonically labeled
compact $n$-packing. The set 
\[
C:=\{x\in(\code\circ\coronas)(P):\codecenter(x)\leq n-2\}\subseteq\codes(S)
\]
 is a fundamental set.
\end{thm}

\begin{proof}
Suppose $C$ is not a fundamental set. As $P$ is a compact $n$-packing,
we have $\{\codecenter(x):x\in C\}=\{0,\ldots,n-2\}$. For $C$ to not
be fundamental there then exists a non-empty set $K\subseteq\{0,1,\ldots,n-2\}$
so that, for every subset $D\subseteq C$ with $\{\codecenter(x):x\in D\}=K$,
we have that $\parenth{\bigcup_{x\in D}\petals(x)}\backslash K=\emptyset.$
Therefore every corona in $P$ with center labeled by an element from
$K$, only has petals labeled by $K$ as well. Therefore only discs
labeled by elements of $K$ occur in $P$, an hence $|\radii(P)|\leq|K|<n$,
contradicting $|\radii(P)|=n$.
\end{proof}
Figure~\ref{fig:fundamental-set} displays an example of a fundamental
set determined from a canonically labeled $3$-packing.

\begin{figure}
\includegraphics[width=1\textwidth]{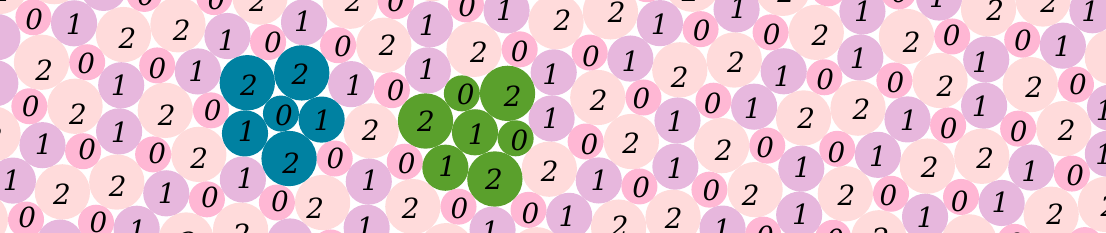}\caption{\label{fig:fundamental-set}An example of a canonically labeled compact
$3$-packing $P$ with fundamental set $\{0:22121,1:212020\}\subseteq(\protect\code\circ\protect\coronas)(P)$
determined from the blue and green coronas (cf. Definition~\ref{def:fundamental}).
The set of coronal codes $\{0:22121,1:212020\}$ determined from the
blue and green coronas also form a $3$-essential set (cf. Definition~\ref{def:Essential-sets}).}
\end{figure}

\section{\label{sec:unique-tight-realizations}Unique tight realizations for
coronal codes from canonically labeled compact packings }

The purpose of this section is to prove Theorem~\ref{thm:fundamental-set-injective-G}.
The main idea is that fundamental sets $C$ (cf. Definition~\ref{def:fundamental})
are such that a map $\mathbf{G}:(0,\infty)^{S}\to\R^{C}$, as defined
in Theorem~\ref{thm:fundamental-set-injective-G}, is injective when
restricted to a certain subset of $(0,\infty)^{S}$. Combining Theorems~\ref{thm:packings-have-fundamental-sets}
and~\ref{thm:fundamental-set-injective-G} then allows us to show,
for any compact $n$-packing $P$, that there exists \emph{exactly}
one element $\Pi_{n}$ that tightly realizes $(\code\circ\coronas)(P)$. 
\begin{lem}
\label{lem:strict-inc-of-angle-sum-in-petal-directions}Let $S$ be
any set of symbols and let $x:=c:p_{0}\ldots p_{m-1}\in\codes(S)$.
For any $\rho\in(0,\infty)^{S}$ and any $v\in[0,\infty)^{S}$ satisfying
both $v(c)=0$ and $v(p)>0$ for at least one petal $p\in\{p_{0},\ldots,p_{m-1}\}$,
the map $[0,\infty)\ni t\mapsto\alpha(x)|_{\rho+tv}$ is strictly
increasing.
\end{lem}

\begin{proof}
Let $t\in[0,\infty)$ be arbitrary. Let $\rho\in(0,\infty)^{S}$ and
let $v\in[0,\infty)^{S}$ satisfy $v(c)=0$ and $v(p)>0$ for at least
one petal $p\in\{p_{0},\ldots,p_{m-1}\}$. Since $v(p)>0$ and $v(c)=0$
we have $p\neq c$. Keeping the definitions of Section~\ref{subsec:Angle-symbols}
in mind, for any $d\in S\backslash\{c\}$ and any $a,b\in S$ we have
$(\partial_{d}c_{b}^{a})|_{\rho+tv}\geq0$ and therefore $(\partial_{d}\alpha(x))|_{\rho+tv}\geq0$.
Further, there exists some $s\in S$ so that the angle symbol $c_{s}^{p}$
occurs at least once as a term in the angle sum $\alpha(x)$. Therefore,
since $\rho+tv\in(0,\infty)^{S}$, we have $(\partial_{p}\alpha(x))|_{\rho+tv}>0$.
Remembering that $v(c)=0$, obtain
\begin{align*}
\sum_{s\in S}(\partial_{s}\alpha(x))|_{\rho+tv}\ v(s) & =\sum_{s\in S\backslash\{c\}}(\partial_{s}\alpha(x))|_{\rho+tv}\ v(s)\\
 & \geq(\partial_{p}\alpha(x))|_{\rho+tv}\ v(p)\\
 & >0.
\end{align*}
Now notice that the directional derivative of the map $(0,\infty)^{S}\ni\sigma\mapsto\alpha(x)|_{\sigma}$
at $\rho+tv$ in the direction of $v$ is a positive scalar multiple
of the left of the above inequality and hence this directional derivative
is strictly positive. Therefore the map $[0,\infty)\ni t\mapsto\alpha(x)|_{\rho+tv}$
is strictly increasing.
\end{proof}
\begin{thm}
\label{thm:fundamental-set-injective-G}Let $n\in\N$ with $n\geq2$
and $S:=\{0,\ldots,n-1\}$. Let $C\subseteq\codes(S)$ be a fundamental
set. The map $\mathbf{G}:(0,\infty)^{S}\to\R^{C}$ defined as
\[
\mathbf{G}(\rho):=(\alpha(x)|_{\rho})_{x\in C}\qquad(\rho\in(0,\infty)^{S}).
\]
is injective when restricted to the set $\{\tau\in(0,\infty)^{S}:\tau(n-1)=1\}$.
\end{thm}

\begin{proof}
Take $\rho,\sigma\in\{\tau\in(0,\infty)^{S}:\tau(n-1)=1\}$ and assume
that $\mathbf{G}(\rho)=\mathbf{G}(\sigma)$, but suppose that $\rho\neq\sigma$.
Hence, by exchanging the roles of $\rho$ and $\sigma$ if necessary,
we are assured that there exists some index $j\in\{0,\ldots,n-2\}$
so that $\rho(j)<\sigma(j)$. 

Define $t_{0}:=\sup\{t\in(0,1):\forall j\in S,\ t\sigma(j)<\rho(j)\}<1$.
As the values of realized angle symbols are invariant under positive
scaling of the realizer, it is readily seen for all $\lambda\in(0,\infty)$
that $\mathbf{G}(\sigma)=\mathbf{G}(\lambda\sigma)$, and in particular,
we have $\mathbf{G}(\sigma)=\mathbf{G}(t_{0}\sigma)$. The set $J:=\{j\in S:t_{0}\sigma(j)=\rho(j)\}$
is non-empty, otherwise $t_{0}$ cannot be the supremum of the set
$\{t\in(0,1):\forall j\in S,\ t\sigma(j)<\rho(j)\}$. Further, since
$t_{0}<1$ we have $n-1\notin J$ and hence $S\backslash J$ is also
non-empty. Since $C$ is a fundamental set, there exists a code $x\in C$,
which we fix, so that $c:=\codecenter(x)\in J$ and $\petals(x)\backslash J$
is not empty. Fix a symbol $p\in\petals(x)\backslash J$.

Define $v:=\rho-t_{0}\sigma\in[0,\infty)^{S}$. The support of $v$
is exactly the set $S\backslash J$ and $v$ takes on positive values
on $S\backslash J$ so that $v(p)>0$ and $v(c)=0$. By Lemma~\ref{lem:strict-inc-of-angle-sum-in-petal-directions},
the map $[0,\infty)\ni t\mapsto\alpha(x)|_{\rho+tv}$ is strictly
increasing. Therefore, since $\mathbf{G}(\sigma)=\mathbf{G}(t_{0}\sigma),$
we have
\[
\alpha(x)|_{\sigma}=\alpha(x)|_{t_{0}\sigma}=\alpha(x)|_{t_{0}\sigma+0v}<\alpha(x)|_{t_{0}\sigma+1v}=\alpha(x)|_{\rho}.
\]
But this contradicts the equality $\alpha(x)|_{\sigma}=\alpha(x)|_{\rho}$
obtained from the assumption $\mathbf{G}(\rho)=\mathbf{G}(\sigma)$.
Therefore the supposition $\rho\neq\sigma$ is false, and therefore
we obtain $\rho=\sigma$. We conclude that $\mathbf{G}$ is injective
when restricted to $\{\tau\in(0,\infty)^{S}:\tau(n-1)=1\}$.
\end{proof}
\begin{thm}
\label{thm:packings-determine-exactly-one-realization}Let $n\in\N$
with $n\geq2$ and set $S:=\{0,\ldots,n-1\}$. Let $P$ be any canonically
labeled compact packing with $\abs{\radii(P)}=n$ and $\max\,\radii(P)=1$.
The canonical realizer of $P$ is the unique element from $\{\tau\in(0,\infty)^{S}:\tau(n-1)=1\}$
that tightly realizes the set $(\code\circ\coronas)(P)\subseteq\codes(S)$.
\end{thm}

\begin{proof}
By Theorem~\ref{thm:packings-have-fundamental-sets}, the set $C:=\{x\in\code\circ\coronas(P):\codecenter(x)\leq n-2\}$
is a fundamental set. By Theorem~\ref{thm:fundamental-set-injective-G},
the map $\mathbf{G}:(0,\infty)^{S}\to\R^{C}$ defined as
\[
\mathbf{G}(\rho):=(\alpha(x)|_{\rho})_{x\in C},\qquad(\rho\in(0,\infty)^{S})
\]
is injective when restricted to $\{\tau\in(0,\infty)^{S}:\tau(n-1)=1\}$.
The canonical realizer $\rho$ of $P$ is an element of $\{\tau\in(0,\infty)^{S}:\tau(n-1)=1\}$
and satisfies $\mathbf{G}(\rho)_{x}=\alpha(x)|_{\rho}=2\pi$ for all
$x\in C$, so $\rho$ is the unique element in $\{\tau\in(0,\infty)^{S}:\tau(n-1)=1\}$
that tightly realizes $C$. Since every realizer that tightly realizes
$(\code\circ\coronas)(P)$ also tightly realizes $C$, the canonical
realizer $\rho$ of $P$ is the unique element of $\{\tau\in(0,\infty)^{S}:\tau(n-1)=1\}$
that tightly realizes $(\code\circ\coronas)(P)$.
\end{proof}

\section{\label{sec:The-bootstrapping-theorem}The bootstrapping lemma}

In this section we prove what we call The Bootstrapping Lemma (Lemma~\ref{lem:bootstrap})
which should be read as a generalization of Lemma~\ref{lem:Fernique-lemma}.
The proof of Lemma~\ref{lem:bootstrap} leverages the definition
of a fundamental set $C$ to allow for the comparison of ratios of
values from two different realizations $\rho$ and $\sigma$ satisfying
$\alpha(x)|_{\rho}\leq2\pi\leq\alpha(x)|_{\sigma}$ for all $x\in C$. 
\begin{lem}[The Bootstrapping Lemma]
\label{lem:bootstrap}Let $n\in\N$ with $n\geq2$ and $S:=\{0,\ldots,n-1\}$.
Let $C\subseteq\codes(S)$ be a fundamental set. Let $\rho,\sigma\in(0,\infty)^{S}$
be maps such that for all $x\in C$ we have $\alpha(x)|_{\rho}\leq2\pi\leq\alpha(x)|_{\sigma}.$
Then 
\[
\frac{\sigma(n-2)}{\sigma(n-1)}\leq\frac{\rho(n-2)}{\rho(n-1)}.
\]
\end{lem}

\begin{proof}
If, for some $t\in(0,\infty)$, it holds that $\rho=t\sigma$, then
we immediately obtain 
\[
\frac{\sigma(n-2)}{\sigma(n-1)}=\frac{t\sigma(n-2)}{t\sigma(n-1)}=\frac{\rho(n-2)}{\rho(n-1)}.
\]
We therefore assume, for all $t\in(0,\infty)$, that $\rho\neq t\sigma.$ 

Define $t_{0}:=\sup\{t\in(0,\infty):\forall j\in S,\ t\sigma(j)<\rho(j)\}$
and $J:=\{j\in S:t_{0}\sigma(j)=\rho(j)\}$. The set $J$ is non-empty,
else $t_{0}$ cannot be the supremum of the set $\{t\in(0,\infty):\forall j\in S,\ t\sigma(j)<\rho(j)\}$.
Furthermore, since $\rho\neq t\sigma$ for all $t\in(0,\infty)$,
the set $J$ does not equal the whole of $S$. 

We consider the two cases $n-1\in J$ and $n-1\notin J$.

If $n-1\in J$, then by definition of $t_{0}$ and $J$, we have $t_{0}\sigma(n-1)=\rho(n-1)$
and $t_{0}\sigma(n-2)\leq\rho(n-2)$. Hence 
\[
\frac{\sigma(n-2)}{\sigma(n-1)}=\frac{t_{0}\sigma(n-2)}{t_{0}\sigma(n-1)}=\frac{t_{0}\sigma(n-2)}{\rho(n-1)}\leq\frac{\rho(n-2)}{\rho(n-1)}.
\]

The case $n-1\notin J$ remains. We show that this case cannot occur.
Suppose $n-1\notin J$, so that $J\subseteq\{0,\ldots,n-2\}$. Since
$C$ is fundamental, there exists a code $x\in C$ so that $c:=\codecenter(x)\in J$
but with a petal $p\in\petals(x)$ so that $p\notin J$. Since realized
angle sums are invariant under positive scaling of the realizer, we
have $\alpha(x)|_{\sigma}=\alpha(x)|_{t_{0}\sigma}$. Define $v:=\rho-t_{0}\sigma$.
The support of $v$ is exactly the set $S\backslash J$ and $v$ takes
on positive values on $S\backslash J$ so that $v(c)=0$ and $v(p)>0$.
By Lemma~\ref{lem:strict-inc-of-angle-sum-in-petal-directions},
the map $[0,\infty)\ni t\mapsto\alpha(x)|_{\rho+tv}$ is strictly
increasing. However this yields the absurdity
\[
2\pi\leq\alpha(x)|_{\sigma}=\alpha(x)|_{t_{0}\sigma}=\alpha(x)|_{t_{0}\sigma+0v}<\alpha(x)|_{t_{0}\sigma+1v}=\alpha(x)|_{\rho}\leq2\pi,
\]
from which we conclude that the supposition $n-1\notin J$ is false.
\end{proof}
We briefly show how Lemma~\ref{lem:bootstrap} implies Lemma~\ref{lem:Fernique-lemma}.
For every canonically labeled corona $Q$ with $\abs{\radii(Q)}=2$
with smaller disc as center, let $\sigma_{Q}\in(0,\infty)^{\{0,1\}}$
denote the canonical realizer of $Q$. Furthermore, there are only
finitely many such coronas (modulo rigid motion and scaling). Take
any canonically labeled compact $3$-packing $P$ with canonical realizer
$\rho$. Obviously $\alpha(x)|_{\rho}=2\pi$ for all $x\in(\code\circ\coronas)(P)$.
Take any coronal code $x\in(\code\circ\coronas)(P)\subseteq\codes(\{0,1,2\})$
that has center $0$ and that has a petal that is not equal to $0$.
Construct the code $y\in\codes(\{0,1\})$ from $x$ by replacing every
symbol $2$ occurring in $x$ with a $1$. The code $y$ is noticed
to coincide with a code determined by a canonically labeled corona
$Q'$ with $\abs{\radii(Q')}=2$ with smaller disc as center. But
now $\alpha(y)|_{\rho}\leq2\pi=\alpha(y)|_{\sigma_{Q'}}$. Therefore,
with fundamental set $C:=\{y\}$, by Lemma~\ref{lem:bootstrap},
\[
\frac{\sigma_{Q'}(0)}{\sigma_{Q'}(1)}\leq\frac{\rho(0)}{\rho(1)}.
\]
As there exist only finitely many canonically labeled coronas $Q$
with $\abs{\radii(Q)}=2$ having small disc as center (modulo rigid
motion and scaling), taking $K:=\min_{Q}\{\sigma_{Q}(0)/\sigma_{Q}(1)\}>0$,
we obtain 
\[
0<K\leq\frac{\sigma_{Q'}(0)}{\sigma_{Q'}(1)}\leq\frac{\rho(0)}{\rho(1)}
\]
with $K$ independent of the compact $3$-packing $P$, establishing
Lemma~\ref{lem:Fernique-lemma}.

\section{\label{sec:Essential-sets-of}Essential sets of coronal codes and
proof of the main result}

We now turn to proving the main result of this paper. We introduce
what we call $n$-essential sets of coronal codes (Definition~\ref{def:Essential-sets}).
We firstly show, for every $n\in\N$ with $n\geq2$, that the set
of all $n$-essential sets, $\cal E_{n}$, is finite (cf. Corollary~\ref{cor:set-of-essential-sets-are-finite}).
This proceeds through a strong induction argument in Lemma~\ref{lem:Em-finite-implies-En-finite}
which utilizes the The Bootstrapping Lemma (Lemma~\ref{lem:bootstrap})
to place a universal bound on the length of all coronal codes that
can occur as elements of any $n$-essential set. Finally we show every
compact $n$-packing determines one of the finitely many $n$-essential
sets in $\cal E_{n}$ (cf. Theorem~\ref{thm:compact-packings-have-n-essential-sets})
and that every element in $\cal E_{n}$ determines at most one element
of $\Pi_{n}$ so that $\abs{\Pi_{n}}\leq\abs{\cal E_{n}}$ (Corollary~\ref{cor:main}).
This will then prove the main result of this paper, Theorem~\ref{thm:main-thm}.
\begin{defn}
\label{def:Essential-sets}Let $n\in\N$ with $n\geq2$ and $S:=\{0,\ldots,n-1\}$.
We will say that a set $C\subseteq\codes(S)$ is an \emph{$n$-essential
}set if $C$ satisfies both of the following conditions:
\begin{enumerate}
\item The set $C$ is a fundamental set.
\item There exist monotone maps $\rho,\sigma\in(0,\infty)^{S}$ so that,
for all $x\in C$, we have $\alpha(x)|_{\rho}\leq2\pi\leq\alpha(x)|_{\sigma}.$
\end{enumerate}
We denote the set of all $n$-essential sets by $\cal E_{n}\subseteq2^{\codes(S)}$.
\end{defn}

\begin{defn}
Let $S$ be a totally ordered set. Let $E$ be any formal expression
containing symbols from $S$. For $k\in S$, by the formal expression
$E\downarrow_{k}$ we mean the formal expression constructed from
$E$ with every symbol that occurs in $E$ that is strictly larger
than $k$ being replaced by $k$.
\end{defn}

\begin{lem}
\label{lem:shrinking-essential-set-remains-essential}Let $n\in\N$
with $n\geq2$ and $S:=\{0,\ldots,n-1\}$. If $C\subseteq\codes(S)$
is $n$-essential, then for every $m\in\{2,\ldots,n-1\}$ with $S_{m}:=\{0,\ldots,m-1\}$,
the set $C':=\{x\downarrow_{m-1}:x\in C,\ \codecenter(x)\leq m-2\}\subseteq\codes(S_{m})$
is $m$-essential.
\end{lem}

\begin{proof}
Assume that $C$ is $n$-essential. By definition, $C$ is fundamental
and there exist monotone maps $\rho,\sigma\in(0,\infty)^{S}$ so that,
for all $x\in C$, we have $\alpha(x)|_{\rho}\leq2\pi\leq\alpha(x)|_{\sigma}.$

Let $m\in\{2,\ldots,n-1\}$, and define $C':=\{x\downarrow_{m-1}:x\in C,\ \codecenter(x)\leq m-2\}\subseteq\codes(S_{m}).$
We firstly claim that $C'$ is a fundamental set. That the set of
centers of codes from $C'$ is $\{0,\ldots,m-2\}$ follows from construction.
Let 
\[
K\subseteq\{0,\ldots,m-2\}\subseteq\{0,\ldots,n-2\}
\]
be any non-empty set. As $C$ is fundamental, there exists a subset
$D\subseteq C$ so that $\{\codecenter(x):x\in D\}=K$ with $\parenth{\bigcup_{x\in D}\petals(x)}\backslash K\neq\emptyset$.
We define $D':=\{x\downarrow_{m-1}:x\in D\}\subseteq C'$. Since,
for all $k\in K$ we have $k<m-1$, then 
\[
\{\codecenter(x):x\in D'\}=\{\codecenter(x):x\in D\}=K.
\]
Further for any $p\in\parenth{\bigcup_{x\in D}\petals(x)}\backslash K\neq\emptyset$
we either have that have that $p\geq m-1$ or that $p<m-1$ with $p\notin K$.
In the case that $p\geq m-1$, we have that $p\downarrow_{m-1}=m-1\notin K$
so that $p\downarrow_{m-1}=m-1\in\parenth{\bigcup_{x\in D'}\petals(x)}\backslash K$.
In the case that $p<m-1$ with $p\notin K$, we have $p\downarrow_{m-1}=p\notin K$
so that $p\downarrow_{m-1}=p\in\parenth{\bigcup_{x\in D'}\petals(x)}\backslash K$.
In both cases $\parenth{\bigcup_{x\in D'}\petals(x)}\backslash K\neq\emptyset$.
Therefore $C'$ is fundamental.

We define $\rho',\sigma'\in(0,\infty)^{S_{m}}$ as 
\[
\rho'(s):=\begin{cases}
\rho(s) & s<m-1\\
\rho(m-1) & s=m-1
\end{cases}\qquad(s\in S_{m})
\]
and 
\[
\sigma'(s):=\begin{cases}
\sigma(s) & s<m-1\\
\sigma(n-1) & s=m-1
\end{cases}\qquad(s\in S_{m}).
\]
Since $\rho$ and $\sigma$ are both monotone, so are $\sigma'$ and
$\rho'$. Further, for all $y\in C'$ take $x\in C$ so that $y=x\downarrow_{m-1}$,
we have 
\[
\alpha(y)|_{\rho'}=\alpha(x\downarrow_{m-1})|_{\rho'}\leq\alpha(x)|_{\rho}\leq2\pi\leq\alpha(x)|_{\sigma}\leq\alpha(x\downarrow_{m-1})|_{\sigma'}=\alpha(y)|_{\sigma'}\ .
\]
We therefore conclude that $C'$ is $m$-essential.
\end{proof}
\begin{lem}
\label{lem:Em-finite-implies-En-finite}Let $n\in\N$ with $n\geq3$.
If, for all $m\in\{2,\ldots,n-1\}$, the set $\cal E_{m}$ is finite,
then the set $\cal E_{n}$ is finite.
\end{lem}

\begin{proof}
Set $S:=\{0,\ldots,n-1\}.$ Assume for all $m\in\{2,\ldots,n-1\}$,
that $\abs{\cal E_{m}}<\infty$. For every $m\in\{2,\ldots,n-1\}$
define $S_{m}:=\{0,\ldots,m-1\}$ and for every $D\in\cal E_{m}$
fix a monotone map $\sigma_{D}\in(0,\infty)^{S_{m}}$ so that, for
all $x\in D$, we have $2\pi\leq\alpha(x)|_{\sigma_{D}}$. For $m\in\{2,\ldots,n-1\}$
define 
\[
K_{m-2}:=\min_{D\in\cal E_{m}}\frac{\sigma_{D}(m-2)}{\sigma_{D}(m-1)}\in(0,1].
\]

Now, take any $n$-essential set $C\in\cal E_{n}$. Let $\rho\in(0,\infty)^{S}$
be any monotone map such that, for all $x\in C$, we have $\alpha(x)|_{\rho}\leq2\pi$.
By Lemma~\ref{lem:shrinking-essential-set-remains-essential}, for
every $m\in\{2,\ldots,n-1\}$, the set $C_{m}:=\{x\downarrow_{m-1}:x\in C,\ \codecenter(x)\leq m-2\}$
is $m$-essential so that $C_{m}\in\cal E_{m}$. For every $m\in\{2,\ldots,n-1\}$,
define $\rho_{m}$ as the restriction of $\rho$ to $S_{m}$. Then,
for every $y\in C_{m}$ taking $x\in C$ so that $y=x\downarrow_{m-1}$,
since $\rho_{m}$ is monotone, we obtain 
\[
\alpha(y)|_{\rho_{m}}\leq\alpha(x)|_{\rho}\leq2\pi\leq\alpha(y)|_{\sigma_{C_{m}}}.
\]
Hence, by Lemma~\ref{lem:bootstrap}, for every $m\in\{2,\ldots,n-1\}$,
we have
\[
0<K_{m-2}=\min_{D\in\cal E_{m}}\frac{\sigma_{D}(m-2)}{\sigma_{D}(m-1)}\leq\frac{\sigma_{C_{m}}(m-2)}{\sigma_{C_{m}}(m-1)}\leq\frac{\rho_{m}(m-2)}{\rho_{m}(m-1)}=\frac{\rho(m-2)}{\rho(m-1)}.
\]
Therefore, from
\[
0<K_{0}\leq\frac{\rho(0)}{\rho(1)},\quad0<K_{1}\leq\frac{\rho(1)}{\rho(2)},\qquad\ldots,\qquad0<K_{n-3}\leq\frac{\rho(n-3)}{\rho(n-2)},
\]
we surmise that
\[
0<\parenth{\prod_{j=0}^{n-3}K_{j}}\leq\frac{\rho(0)}{\rho(n-2)},
\]
with $\parenth{\prod_{j=0}^{n-3}K_{j}}\in(0,1]$. We define 
\[
\kappa(s):=\begin{cases}
\parenth{\prod_{j=0}^{n-3}K_{j}} & s=0\\
1 & s=n-2,
\end{cases}\quad(s\in\{0,n-2\}).
\]
Since $\rho$ is monotone, we have $(n-2)_{0}^{0}|_{\rho}\leq\min\{\ang(x)|_{\rho}:x\in C\}$,
and since $0<\parenth{\prod_{j=0}^{n-3}K_{j}}\leq\rho(0)/\rho(n-2)$,
we obtain 
\[
0<(n-2)_{0}^{0}|_{\kappa}\leq(n-2)_{0}^{0}|_{\rho}\leq\min\{\ang(x)|_{\rho}:x\in C\}.
\]
Therefore the length of any of the codes in $C$ can be at most 
\[
N:=\left\lfloor 2\pi/\parenth{(n-2)_{0}^{0}|_{\kappa}}\right\rfloor .
\]
It is crucial to observe that $N$ is independent of the choice of
$C\in\cal E_{n}$. The cardinality of $\cal E_{n}$ is thus bounded
by the cardinality of the powerset of the finite set 
\[
\{x\in\codes(S):\len(x)\leq N,\ \codecenter(x)\leq n-2\}.
\]
In particular $\abs{\cal E_{n}}<\infty$.
\end{proof}
\begin{lem}
\label{lem:E2-finite}The set $\cal E_{2}$ is finite.
\end{lem}

\begin{proof}
Let $C\in\cal E_{2}$. Since $C$ is fundamental, we have that $C\subseteq\{x\in\codes(\{0,1\}):\codecenter(x)=0\}$
and there exist monotone maps $\rho,\sigma\in(0,\infty)^{\{0,1\}}$
so that for all $x\in C$ 
\[
\alpha(x)|_{\rho}\leq2\pi\leq\alpha(x)|_{\sigma}.
\]
For all $a,b\in\{0,1\}$ we have $\pi/3\leq0_{b}^{a}|_{\rho}$ which
implies that the angle sum $\alpha(x)$ can have at most six terms
and hence the length of any code $x\in C$ is at most six. Therefore
the cardinality of $\cal E_{2}$ is bounded by the cardinality of
powerset of the finite set $\{x\in\codes(\{0,1\}):\len(x)\leq6,\ \codecenter(x)=0\}$.
\end{proof}
\begin{cor}
\label{cor:set-of-essential-sets-are-finite}For every $n\in\N$ with
$n\geq2$, the set $\cal E_{n}$ is finite.
\end{cor}

\begin{proof}
By strong induction and Lemmas~\ref{lem:E2-finite} and~\ref{lem:Em-finite-implies-En-finite},
for every $n\in\N$ with $n\geq2$, the set $\cal E_{n}$ is finite. 
\end{proof}
{}
\begin{thm}
\label{thm:compact-packings-have-n-essential-sets}Let $n\in\N$ with
$n\geq2$ and $S:=\{0,\ldots,n-1\}$. Let $P$ be a canonically labeled
compact $n$-packing. Then there exists a subset of $C\subseteq(\code\circ\coronas)(P)$
that is $n$-essential, i.e., $C\in\cal E_{n}$.
\end{thm}

\begin{proof}
By Theorem~\ref{thm:packings-have-fundamental-sets}, there exists
a fundamental set $C\subseteq(\code\circ\coronas)(P)$. With $\rho$
the canonical realizer of $P$, for all $x\in C$ we have $\alpha(x)|_{\rho}=2\pi$,
so taking $\sigma:=\rho$, we obtain $\alpha(x)|_{\rho}=2\pi=\alpha(x)|_{\sigma}$
for all $x\in C$. Therefore $C\in\cal E_{n}.$
\end{proof}
Figure~\ref{fig:fundamental-set} displays an example of a $3$-essential
set determined from a canonically labeled $3$-packing.
\begin{cor}
\label{cor:main}For every $n\in\N$ with $n\geq2$, the cardinality
of the set 
\[
\Pi_{n}:=\set{\radii(P)}{\begin{array}{c}
P\text{ a compact \ensuremath{n}-packing with }\\
\max\,\radii(P)=1
\end{array}}
\]
 is at most the cardinality of $\cal E_{n}$, and therefore $\Pi_{n}$
is finite.
\end{cor}

\begin{proof}
Let $n\in\N$ with $n\geq2$ and set $S:=\{0,\ldots,n-1\}$. By Theorem~\ref{thm:fundamental-set-injective-G},
for every element $C$ of the finite set $\cal E_{n}$, there is at
most one (but perhaps no) element $\rho\in\{\tau\in(0,\infty)^{S}:\tau(n-1)=1\}$
so that, for all $x\in C$, we have $\alpha(x)|_{\rho}=2\pi$. 

On the other hand, let $P$ be any compact $n$-packing with $\max\,\radii(P)=1$.
By Theorem~\ref{thm:compact-packings-have-n-essential-sets}, there
exists a subset $C\subseteq(\code\circ\coronas)(P)$ that is an element
of $\cal E_{n}$ and, by Theorem~\ref{thm:packings-determine-exactly-one-realization},
the canonical realizer $\rho:S\to(0,1]$ of $P$ is an element of
$\{\tau\in(0,\infty)^{S}:\tau(n-1)=1\}$ so that, for all $x\in C$,
we have $\alpha(x)|_{\rho}=2\pi$. 

We therefore obtain $\abs{\Pi_{n}}\leq\abs{\cal E_{n}}<\infty$.
\end{proof}
The previous corollary proves Theorem~\ref{thm:main-thm}, as stated
in the introduction.
\begin{rem}
\label{rem:why-no-quantitative-bounds}For $n\in\N$, from Corollary~\ref{cor:main}
a quantitative bound for $\abs{\Pi_{n}}$ is the finite number $\abs{\cal E_{n}}$.
However, for a fixed $n\in\N$, given exact knowledge of the sets
$\cal E_{2},\ldots,\cal E_{n-1}$, Lemma~\ref{lem:Em-finite-implies-En-finite}
only places a bound on the size of the set $\cal E_{n}$, but gives
no information on the elements of $\cal E_{n}$. Specifically, Lemma~\ref{lem:Em-finite-implies-En-finite}
and its proof does not give information on the monotone maps $\rho$
and $\sigma$ that exist for every element of $\cal E_{n}$ (Definition~\ref{def:Essential-sets}).
But as seen in the proof of Lemma~\ref{lem:Em-finite-implies-En-finite}
this information is required to be able to determine a bound on the
size of the next set $\cal E_{n+1}$ in the sequence and, by extension,
a bound on the size of $\Pi_{n+1}$. For this reason, determining
a quantitative bound for $\abs{\Pi_{n}}$ for $n\in\N$ using the
ideas presented in this paper will at least require computing monotone
maps $\rho$ and $\sigma$, as in Definition~\ref{def:Essential-sets},
that exist for each element in each of the sets $\cal E_{2},\ldots,\cal E_{n-1}$.
\end{rem}